\documentclass[12pt]{amsart} 
\usepackage[utf8]{inputenc}

\usepackage{amsmath, amssymb, amsfonts, amsbsy, amsthm, latexsym, stmaryrd, mathtools, fullpage, enumerate, ytableau, shuffle, multirow, mathdots, enumerate, tikz, tikz-cd, verbatim, comment, hhline, arydshln}
\PassOptionsToPackage{hyphens}{url}\usepackage{hyperref}\PassOptionsToPackage{hyphens}{url}\usepackage{hyperref}
\usetikzlibrary{calc,matrix,decorations.pathreplacing}

\newtheorem{thm}{Theorem}[section]
\newtheorem{lem}[thm]{Lemma}
\newtheorem{prop}[thm]{Proposition}

\newtheorem{defn}[thm]{Definition}
\theoremstyle{definition}
\newtheorem{constr}[thm]{Construction}

\newtheorem{rem}[thm]{Remark}

\newtheorem*{defn*}{Definition}
\newtheorem*{thm*}{Theorem}
\newtheorem*{prop*}{Proposition}
\newtheorem*{cor*}{Corollary}

\newtheorem*{question*}{Question}

\newcommand{\N}{\mathbb{N}}
\newcommand{\Z}{\mathbb{Z}}

\newcommand{\R}{\mathbb{R}}

\newcommand{\T}{\mathbb{T}}

\newcommand{\Rm}{\operatorname{Rm}}
\newcommand{\Ric}{\operatorname{Ric}}

\newcommand{\Lip}{\operatorname{Lip}}

\newcommand{\tr}{\operatorname{tr}}

\newcommand{\im}{\operatorname{Im}}

\newcommand{\Ker}{\operatorname{Ker}}

\newcommand{\sff}{{\rm II}}

\newcommand{\cA}{\mathcal A}

\newcommand{\cC}{\mathcal C}

\newcommand{\cG}{\mathcal G}
\newcommand{\cH}{\mathcal H}

\newcommand{\cL}{\mathcal L}

\newcommand{\cQ}{\mathcal Q}
\newcommand{\cR}{\mathcal R}

\newcommand{\cV}{\mathcal V}

\title{A Generalization of the Geroch Conjecture with Arbitrary Ends}
\address{Department of Mathematics, Stanford University, 450 Jane Stanford Way, Bldg
380, Stanford, CA 94305}
\author{Shuli Chen}
\email{shulic@stanford.edu}

\begin{document}

\begin{abstract} Using $\mu$-bubbles, we prove that for $3 \le n \le 7$, the connected sum of a Schoen--Yau--Schick $n$-manifold with an arbitrary manifold does not admit a complete metric of positive
scalar curvature.

When either $3 \le n \le 5$, $1 \le m \le n-1$ or $6 \le n \le 7$, $m \in \{1, n-2, n-1\}$, we also show the connected sum $(M^{n-m}\times \T^m) \# X^n$ where $X$ is an arbitrary manifold
does not admit a metric of positive $m$-intermediate curvature. Here $m$-intermediate curvature is a new notion of curvature introduced by Brendle, Hirsch and Johne interpolating between Ricci and scalar curvature.
\end{abstract}

\maketitle

\section{Introduction}
The well-known Geroch conjecture asks whether the torus $\T^n$ admits a metric of positive scalar curvature. A negative answer to this conjecture was given by Schoen and Yau for $3 \le n \le 7$ using
minimal hypersurfaces via the inductive descent method \cite{SY:descent}, and by Gromov and Lawson for all dimensions using spinors \cite{gromov1983positive}.
This result has had several important consequences,
including Schoen-Yau’s proof of the positive mass theorem in
general relativity \cite{schoen1979proof, schoen1989variational, SY:sing-PMT} and Schoen’s resolution of the Yamabe problem \cite{schoen1984conformal}.

The Geroch conjecture has been generalized in various ways. For instance, Chodosh and Li \cite{soapbubble}
proved the Geroch conjecture with arbitrary ends for $3 \le n \le 7$ via the $\mu$-bubble technique; namely, they proved for any $n$-manifold $X$, the connected sum $\T^n\# X$ does not admit a complete metric of positive scalar curvature. The case $n= 3$ was also obtained independently by Lesourd, Unger, and Yau \cite{lesourd2020positive}.
Recently, in the spin setting, Wang and Zhang \cite{wang2022generalized} showed that for arbitrary $n$ and any spin $n$-manifold $X$, the connected sum $\T^n\# X$ admits no complete metric of
positive scalar curvature. 
Using a similar argument, Chodosh and Li \cite{soapbubble} further extended their result to manifolds of
the form $(M^{n-1} \times S^1) \#X$, where $3\le n \le 7$, $M$ is a Schoen--Yau--Schick manifold and $X$ is arbitrary. Here we recall the definition of a Schoen--Yau--Schick manifold:
\begin{defn}[Schoen--Yau--Schick manifold, \cite{SY:descent,schick1998counterexample,SY:sing-PMT,gromov2018metric}]
An orientable closed manifold $M^n$ is called a Schoen--Yau--Schick manifold (abbreviated as SYS manifold), if there are nonzero cohomology
classes $\beta_1, \beta_2, \dots, \beta_{n-2}$ in $H^1(M;\Z)$ such that the homology class
$[M] \frown (\beta_1 \smile \beta_2 \smile \dots \smile \beta_{n-2}) \in H_2(M;\Z)$ is non-spherical, that is, it does not lie in the image of the Hurewicz homomorphism $\pi_2(M) \to H_2(M;\Z)$.
\end{defn}
In particular, the torus is an SYS manifold. SYS manifolds were first considered by Schoen and Yau in \cite{SY:descent}, where they proved that SYS manifolds of dimension at most 7 do not admit metrics of positive scalar curvature via the inductive descent argument. Later, Schick \cite{schick1998counterexample} constructed an SYS manifold as a counterexample to the unstable Gromov--Lawson--Rosenberg conjecture. 

\begin{thm}\cite{soapbubble}\label{thm: SYS times S1}
Let $3 \le n \le 7$, and let $M^{n-1}$ be a Schoen--Yau--Schick manifold. For
any $n$-manifold $X$, the connected sum $(M^{n-1} \times S^1)\#X$ does not admit a
complete metric of positive scalar curvature.
\end{thm} 

The presence of the $S^1$ factor in the preceding theorem is to pass to an appropriate covering space in order to apply the $\mu$-bubble technique introduced by Gromov in \cite{gromov1996positive}. In this paper, we show that we can pass to the infinite cyclic cover obtained by cutting and pasting along a hypersurface (see Theorem~\ref{thm: main construction}), thereby 
obtain a generalization of Chodosh and Li's result as follows:

\begin{thm}\label{thm: SYS}
Let $3 \le n \le 7$, and let $M^n$ be a Schoen--Yau--Schick manifold. For
any $n$-manifold $X$, the connected sum $M\#X$ does not admit a
complete metric of positive scalar curvature.
\end{thm} 
We note that a version of this result was obtained by Lesourd, Unger, and Yau \cite{lesourd2020positive} for $n=3$ and $4 \le n \le 7$ with certain additional technical hypothesis on $M \# X$.

In another direction to generalize the Geroch conjecture, Brendle, Hirsch, and Johne \cite{brendle2022generalization} defined a family of curvature conditions called \emph{$m$-intermediate curvature}, which reduces to Ricci curvature when $m = 1$ and to scalar curvature when $m = n - 1$. The precise definition is as follows:

\begin{defn}[$m$-intermediate curvature, \cite{brendle2022generalization}] Suppose $(N^n, g)$ is a Riemannian manifold.
Let $\Rm_N(X, Y,Z,W) = -g(\nabla_X \nabla_Y Z - \nabla_Y \nabla_X Z - \nabla_{[X,Y]}Z,W)$ denote the Riemann curvature tensor. Let $1 \le m \le n-1$. 
For every orthonormal basis $\{e_1, \dots, e_n\}$ of $T_p N$, we define
$$\cC_m(e_1, \dots, e_m) := \sum_{p=1}^m\sum_{q=p+1}^n 
  \Rm_N(e_p, e_q, e_p, e_q).
$$
Let 
$$\cC_m(p) : = \min \{\mathcal{C}_m(e_1, \dots, e_m) \mid \{e_1, \dots, e_n\} \text{ is an orthonormal basis of } T_p N\}.$$
Let $C \in \R$. Then we say $(N^n, g)$ has $m$-intermediate curvature $> C$ at $p \in N$, if $\cC_m(p) > C$.  We say $(N^n,g)$ has $m$-intermediate curvature $> C$, if it has $\cC_m(p) > C$ for all $p \in N$.
\end{defn}
In particular, at any $p \in N$, sectional curvature $\sec > 0$ implies $\cC_m(p) > 0$, which in turn implies scalar curvature $R >0$. On the other hand, $\Ric > 0$ at $p$ doesn't necessarily imply $\cC_m(p) > 0$ for $2 \le m \le n-2$.

Brendle, Hirsch, and Johne investigated topological obstructions to positive $m$-intermediate curvature and proved the following result. 

\begin{thm}{\cite[Theorem 1.5]{brendle2022generalization}}\label{thm: BHJ}
Let $3 \le n \le 7$ and $1 \le m \le n-1$. Let $N^n$ be a closed manifold of dimension $n$, and suppose that there exists a closed manifold $M^{n-m}$ and a map $F: N^n \to M^{n-m} \times \T^m$ with non-zero degree. Then the manifold $N^n$ does not admit a metric of positive $m$-intermediate curvature.  
\end{thm}

There are two reasons for the presence of the dimensional constraint $n \le 7$ in Brendle--Hirsch--Johne's result. The first reason comes from the regularity theory of stable minimal hypersurfaces in geometric measure theory. The second reason is that their proof requires some algebraic quantity, $m^2 - mn + 2n -2$, to be nonnegative. Kai Xu \cite{xu2023dimension} demonstrated the optimality of the dimensional constraint $n \le 7$ by constructing concrete counterexamples. Namely, if $m^2 - mn + 2n -2 < 0$, then $S^{n-m}\times \T^m$ admits a metric of positive $m$-intermediate curvature.  

Chu--Kwong--Lee \cite{chu2022rigidity} proved a corresponding rigidity statement for non-negative $m$-intermediate curvature when $n \le 5$, which was extended to $n=6$ by Xu \cite{xu2023dimension}. Again, the dimensional constraint $n \le 6$ was shown by Xu to be optimal.

In this paper, we also apply the $\mu$-bubble technique to obtain the following generalization of Brendle--Hirsch--Johne's result to arbitrary ends:

\begin{thm}\label{thm: intermediate curvature}
Assume either $3 \le n \le 5$, $1 \le m \le n-1$ or $6 \le n \le 7$, $m \in \{1, n-2, n-1\}$.
Let $N^n$ be a closed manifold of dimension $n$, and suppose that there exists a closed manifold $M^{n-m}$ and a map $F: N^n \to M^{n-m} \times \T^{m}$ with non-zero degree. Then for any $n$-manifold $X$, the connected sum $N \#X$ does not admit a complete metric of positive $m$-intermediate curvature.
\end{thm} 

For example, this implies that a punctured manifold of the form $(M^{n-m} \times \mathbb{T}^{m}) \setminus \{\text{point}\}$ does not admit a complete
metric of positive $m$-intermediate curvature when $n$ and $m$ are in the given range. 
Notice that we have a gap here; this is because in our proof, we need some extra algebraic quantity involving $m$ and $n$ to be positive (see Lemma~\ref{lem: algebraic} and Remark~\ref{rem: dimension constraint}). It is an interesting question whether the same result still holds when $6 \le n \le 7$ and $2 \le m \le n-3$.

This paper is organized as follows. In Section 2, we give some topological preliminaries. In Section 3, we discuss $\mu$-bubbles and prove a key result, Theorem~\ref{thm: main construction}, which allows us to reduce the non-compact setting to a compact setting. Using this, we give the proof of Theorem ~\ref{thm: SYS} in Section 4 and the proof
of Theorem~\ref{thm: intermediate curvature} in Section 5.

 \subsection*{Acknowledgments.}
The author is grateful for many useful discussions and suggestions of Otis Chodosh. The author is also grateful for helpful conversations with Sven Hirsch on the work \cite{brendle2022generalization}. The author also wants to thank the reviewers for their careful reading of the manuscript and their constructive remarks. The author is partially sponsored by the Ric Weiland Graduate Fellowship at Stanford University.

\section{Topological Preliminaries}
In this section we collect some basic topological facts for later use.
\begin{lem}\label{lem: representable}
Let $M^n$ be a closed connected orientable smooth manifold and let $0 \neq \alpha \in H_{n-1}(M;\Z)$ be a nonzero homology class. Then $\alpha$ is represented by a closed embedded orientable hypersurface $\Sigma$. 
\end{lem}
\begin{proof}
Notice the space $S^1$ is a $K(\Z,1)$, so $H^1(M; \Z) = [M, S^1]$, where $[M, S^1]$ are homotopy classes of maps from $M$ to $S^1$. Thus we can choose a non-constant smooth map $f: M \to S^1$ representing the Poincar\'e dual of $\alpha$ in $H^1(M; \Z)$. By Sard's theorem we can take the preimage $\Sigma$ of a regular value as a representative of $\alpha$. Then $\Sigma$ is a closed embedded orientable hypersurface by the regular value theorem.

\end{proof}

\begin{lem}
Let $M^n$ be a closed connected orientable manifold and let $\Sigma^{n-1} \subset M$ be an orientable closed embedded connected hypersurface. Then $\Sigma$ is separating (i.e., $M\setminus \Sigma$ is the disjoint union of 2 connected open sets in $M$) if and only if $[\Sigma] = 0$ in $H_{n-1}(M;\Z)$. 
\end{lem}
\begin{proof} 
Suppose $\Sigma$ is non-separating, then $M \setminus \Sigma$ is connected, so there exists a simple loop $S$ in $M$
 which crosses $\Sigma$ transversally in exactly one point. Orient $S$ so that this intersection is positive. Then the oriented intersection number $I([\Sigma], [S])$ equals $1$. Since the oriented intersection number is independent of the representative of the homology class, it follows that $\Sigma$ is homologically nontrivial.

Conversely, suppose that $\Sigma$ separates. 
Let $U$ be a tubular neighborhood of $\Sigma$, and let $V = M \setminus \Sigma = M_+ \cup M_-$. Then $\partial V = \partial M_+ \cup \partial M_- = \Sigma_+ \cup \Sigma_-$, where $\Sigma_+$ has the same orientation as $\Sigma$ while $\Sigma_-$ has the opposite orientation.

Let $i:\Sigma \to M$, $i_1: U \cap V \to U$, $i_2: U \cap V \to V$, $j_1: U \to M$, $j_2: V \to M$ be inclusion maps.
Consider the Mayer-Vietoris sequence in singular homology with $\Z$ coefficients:
\begin{align*}
\dots \to H_{n-1}(U \cap V) \xrightarrow{((i_1)_*, (i_2)_*)} H_{n-1}(U) \oplus H_{n-1}(V) \xrightarrow{(j_1)_* - (j_2)_*} H_{n-1}(M) \to \dots 
\end{align*}

Since $U \cap V$ is homotopy equivalent to a disjoint union $\Sigma_+ \cup \Sigma_-$, we have $H_{n-1}(U \cap V) \cong H_{n-1}(\Sigma_+)\oplus H_{n-1}(\Sigma_-)$, and the map $(i_1)_*: H_{n-1}(U \cap V) \to H_{n-1}(U)$ is given by $H_{n-1}(\Sigma_+)\oplus H_{n-1}(\Sigma_-) \to H_{n-1}(\Sigma), (a[\Sigma_+],b[\Sigma_-]) \mapsto (a-b)[\Sigma]$. Since $\Sigma$ is the boundary of $M_+$, $\Sigma$ is null-homologous in $M_+$ hence also in $V$, showing the map $H_{n-1}(\Sigma) \xrightarrow{\cong} H_{n-1}(U \cap V)  \to H_{n-1}(V)$ is the zero map. Thus the Mayer-Vietoris sequence becomes
\begin{align*}
\dots \to H_{n-1}(\Sigma_+) \oplus H_{n-1}(\Sigma_-) \xrightarrow{((i_1)_*, 0)} H_{n-1}(\Sigma) \oplus H_{n-1}(V) \xrightarrow{i_* - (j_2)_*} H_{n-1}(M) \to \dots   
\end{align*}
Exactness at $H_{n-1}(\Sigma) \oplus H_{n-1}(V)$ shows that $H_{n-1}(\Sigma) \oplus \{0\} = \im \, ((i_1)_*, 0) = \Ker\, (i_* - (j_2)_*)$, so $H_{n-1}(\Sigma) = \Ker i_*$. That is, $i_*: H_{n-1}(\Sigma) \to H_{n-1}(M)$ is the zero map, which means $[\Sigma]$ is trivial in $H_{n-1}(M)$.
\end{proof}

\begin{constr}[$d$-cyclic cover]

Let $M$ be a closed connected $n$-manifold. Let $\Sigma$ be an embedded closed connected non-separating hypersurface in $M$. Given any integer $d \ge 1$ or $d = \infty$, we can obtain a $d$-cyclic cover $\hat{M}$ by cutting and pasting along $\Sigma$. The construction is as follows: 

Cut $M$ along $\Sigma$. Let $\tilde{M} = M \setminus \Sigma$. Then $\tilde{M}$ is a connected manifold with boundary, and $\partial \tilde{M}$ has two components, both diffeomorphic to $\Sigma$. Denote $\partial \tilde{M} = \Sigma_- \cup \Sigma_+$. Let $G = \Z/ d\Z$ when $d$ is finite and $G = \Z$ when $d =\infty$. Let $\tilde{M}_k$, $k \in G$ be $d$ copies of $\tilde{M}$. Glue together $\tilde{M}_k$ along the boundary by gluing the $\Sigma_+$ boundary component of $\tilde{M}_k$ with the $\Sigma_-$ boundary component of $\tilde{M}_{k+1}$. Denote the resulting manifolds by $$\hat{M} = \cup_{k \in G} \tilde{M}_k/ \sim,$$
where the equivalence relation $\sim$ is the gluing we just described. Then $\hat{M}$ is a $d$-cyclic cover of $M$.
\end{constr}

\section{$\mu$-bubbles}
In this section we first collect some general existence and stability results for $\mu$-bubbles. We refer the reader to \cite{soapbubble} for more details, where they considered more generally the warped $\mu$-bubbles. For us, we do not need the warping and we simply take the warping function $u=1$. We then use $\mu$-bubbles to prove a key result, Theorem~\ref{thm: main construction}, which is going to be applied in the proofs of both Theorem~\ref{thm: SYS} and Theorem~\ref{thm: intermediate curvature}.

We begin by fixing some notations. For a Riemannian manifold $(M^n,\overline{g})$ we consider its
Levi-Civita connection $D$ and its 
Riemann curvature tensor $\Rm_M$ given by the formula
\[
 \Rm_M(X,Y,Z,W)
 =
 -\overline{g}( D_X D_Y Z - D_Y D_X Z - D_{[X,Y]} Z ,W)
\]
for vector fields $X,Y,Z,W \in \Gamma(TM)$.

Consider a two-sided embedded submanifold $(\Sigma^{n-1},g)$ with induced metric.
We denote its induced Levi-Civita connection by $D_{\Sigma}$ and its unit normal vector field by $\nu \in \Gamma(N\Sigma)$.
We define its scalar-valued second fundamental form $\sff_{\Sigma}$ by $\sff_{\Sigma}(X,Y):= \langle D_X \nu, Y\rangle$. We define the scalar mean curvature of $\Sigma$ by $H_{\Sigma} = \tr_g \sff_{\Sigma}$. The gradient of a smooth function on $M$ or $\Sigma$
is denoted by $D_M f$ or $D_{\Sigma} f$.

For $n\leq 7$, consider $(M,\overline{g})$, a Riemannian $n$-manifold with boundary, and assume that $\partial M = \partial_- M \cup \partial_+ M$ is a choice of labeling the components of $\partial M$ so that neither of the sets $\partial_\pm M$ are empty. Fix a smooth function $h$ on $\mathring M$ with $h\to \pm \infty$ on $\partial_\pm M$. 
Choose a Caccioppoli set $\Omega_0$ with smooth boundary $\partial\Omega_0 \subset \mathring M$ and $\partial_+ M\subset \Omega_0$. 

Consider the following functional
\begin{equation}\label{problem.variation}
\cA(\Omega)=\cH^{n-1}(\partial^* \Omega) - \int_M (\chi_\Omega-\chi_{\Omega_0})h \, d\cH^n,
\end{equation}
for all Caccioppoli sets $\Omega$ in $M$ with $\Omega\Delta \Omega_0\Subset \mathring M$. We will call a Caccioppoli set $\Omega$ minimizing $\cA$ in this class a $\mu$-bubble. 

The functional $\cA$ was first considered by Gromov in \cite{gromov1996positive}. 
The existence and regularity of a minimizer of $\cA$ among all Caccioppoli sets was
claimed by Gromov in \cite[Section 5.2]{gromov2019fourlectures}, and was rigorously carried out in \cite[Proposition 2.1]{zhu2021width} and also in \cite[Proposition 12]{soapbubble}. We thus record it here.
\begin{prop}[{\cite[Proposition 2.1]{zhu2021width}}{\cite[Proposition 12]{soapbubble}}]\label{prop: existence.regularity}
There exists a smooth minimizer $\Omega$ for $\cA$ such that $\Omega\Delta \Omega_0$ is compactly contained in the interior of $M_1$.
\end{prop}

We next discuss the first and second variation for a $\mu$-bubble.
\begin{lem}[{\cite[Lemma 13]{soapbubble}}]\label{lem: 1st-var}
If $\Omega_t$ is a smooth $1$-parameter family of regions with $\Omega_0 = \Omega$ and normal speed $\psi$ at $t=0$, then 
\[\frac{d}{dt}\cA (\Omega_t)=\int_{\Sigma_t} (H-h)\psi  \, d\cH^{n-1}\]
where $H$ is the scalar mean curvature of $\partial\Omega_t$. In particular, a $\mu$-bubble $\Omega$ satisfies \[
H = h
\]
along $\partial\Omega$. 
\end{lem}
\begin{lem}\label{lem: 2nd-var}
Consider a $\mu$-bubble $\Omega$ with $\partial\Omega = \Sigma$. Assume that $\Omega_t$ is a smooth $1$-parameter family of regions with $\Omega_0 = \Omega$ and normal speed $\psi$ at $t=0$, then $\cQ(\psi):=\frac{d^2}{dt^2}\big|_{t=0}(\cA(\Omega_t))\ge 0$ where $\cQ(\psi)$ satisfies 
\begin{align*}
\cQ(\psi) 
= \int_\Sigma \left(|D_\Sigma \psi|^2  - \big(|\sff_{\Sigma}|^2 + \Ric_{M}(\nu,\nu) + \langle D_{M} h, \nu \rangle\big)\psi^2  
\right)d\cH^{n-1},
\end{align*}
where $\nu$ is the outwards pointing unit normal.
\end{lem}
\begin{proof}
Let $\Sigma_t : = \partial \Omega_t$. By the variation formulas for hypersurfaces (see e.g. \cite[Theorem 3.2]{huisken1999geometric}), we have
\[
\frac{\partial}{\partial t} H_{\Sigma_t} \Big|_{t=0} = -\Delta_\Sigma \psi -   \left( |\sff_{\Sigma}|^2 + \Ric_M(\nu, \nu) \right) \psi,
\] 
Differentiating the first variation and using $H_\Sigma = h$, we thus have
\begin{align*}
 \cQ(\psi) & =  \frac{\partial}{\partial t}\Big|_{t=0} \int_{\Sigma_t} (H-h)\psi  \, d\cH^{n-1} \\
 & =  \int_{\Sigma_0} \left((-\Delta_\Sigma \psi -   \left( |\sff_{\Sigma}|^2 + \Ric_M(\nu, \nu) \right) \psi - \langle D_M h, \psi \nu \rangle\right)\psi  \, d\cH^{n-1} \\
& = \int_\Sigma \left(|D_\Sigma \psi|^2  - \big(|\sff_{\Sigma}|^2 + \Ric_{M}(\nu,\nu) + \langle D_{\Sigma} h, \nu \rangle\big)\psi^2  
\right)d\cH^{n-1}.
\end{align*}
\end{proof}

Below we prove the key result, where we reduce the non-compact case to the compact case via $\mu$-bubbles.

\begin{thm}\label{thm: main construction}
Let $3 \le n \le 7$, and let $1 \le m \le n-1$. Let $M^{n}$ be a closed connected orientable manifold such that 
there exists a closed connected orientable non-separating hypersurface $\Sigma$.
Let $X$ be any $n$-manifold, and consider the connected sum $Y = M \# X$.
Suppose $Y$ admits a complete metric of positive $m$-intermediate curvature.

Then for any number $a>0$, there exists a closed connected orientable Riemannian manifold $(\tilde Y, \tilde g)$, a smooth function $h \in C^\infty(Y)$, and a closed embedded orientable hypersurface $\Lambda^{n-1} \subset \tilde{Y}$ such that
\begin{itemize}
\item $\tilde Y = M' \#_i \tilde{X}_i$, where $M'$ is a finite cyclic covering of $M$ obtained by cutting and pasting along $\Sigma$ and the $\tilde{X}_i$'s are a finite number of closed manifolds.
\item In a neighborhood of $\Lambda$, $\tilde Y$ has positive $m$-intermediate curvature.
\item $p_*[\Lambda] = [\Sigma] \in H_{n-1}(M')$, where $p: \tilde Y \to M'$ is the projection map and $[\Sigma]$ is the homology class represented by any copy of $\Sigma$ in $M'$.
\item On $\Lambda$, we have
$$H=h,$$
$$(\cC_m)_{\tilde{Y}} + ah^2 - 2|D_{\tilde{Y}} h| > 0$$
and 
$$\cQ(\psi) 
= \int_\Lambda \left(|D_\Lambda \psi|^2  - \big(|\sff_{\Lambda}|^2 + \Ric_{\tilde{Y}}(\nu,\nu) + \langle D_{\tilde{Y}} h, \nu \rangle\big)\psi^2 
\right)d\cH^{n-1} \ge 0
$$
for all $\psi \in C^\infty(\Lambda)$. 
\end{itemize}
\end{thm}

\begin{proof}
We follow the approach of \cite[Section 6 and 7]{soapbubble}. Namely, we pass to an appropriate covering space of $Y$, construct a weight function $h$, and apply the $\mu$-bubble technique. The main difference in our case is how to find the covering space and how to modify the construction of the weight function $h$. An illustration of the construction is in Figure \ref{fig: construction}.

Let $Y = M \# X$ be as in the assumption. By taking the orientation double cover of $X$ we can assume $X$ is orientable. Let $p \in M$ be a point such that $B(p) \cap \Sigma = \emptyset$, where $B(p)$ is a small $n$-ball around $p$. Let $M' = M \setminus B(p)$ and $X' = X\setminus B$, where $B$ is a small $n$-ball in $X$.
Then we can take $Y = M \# X= M' \cup X'$, where $M'$ and $X'$ are glued on the boundary sphere.

Suppose $Y$ is endowed with a complete metric of positive $m$-intermediate curvature. By scaling and compactness of $M'$ we can assume $\cC_m >1$ on $M'$.
Let $a>0$ be any number.

\begin{figure}
\includegraphics[width=0.7\textwidth]{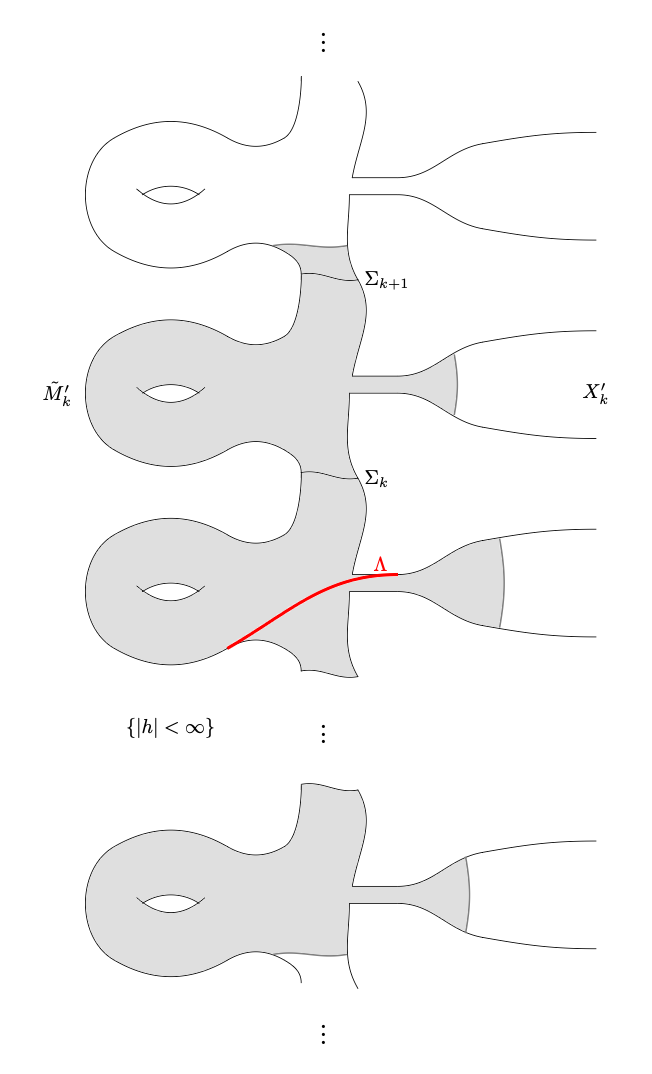}
\caption{Illustration of the construction in the proof of Theorem \ref{thm: main construction}. The manifold is $\hat{Y} = \hat{M} \#_\Z X$, the infinite cyclic cover of $M \# X$ constructed in Step 1. The shaded region is the bounded set $\{|h|< \infty\}$ obtained in Step 2. The red hypersurface is $\Lambda$ obtained in Step 3. }
\label{fig: construction}
\end{figure}

\textbf{Step 1: pass to an infinite cyclic cover by cutting and pasting $M$ along $\Sigma$.}
Cut $M$ along $\Sigma$. Let $\tilde{M} = M \setminus \Sigma$ and $\tilde{M}' = M' \setminus \Sigma = (M\setminus B(p)) \setminus \Sigma$. Then $\tilde{M}$ is a connected manifold with boundary, and  $\partial \tilde{M}$ has two components, both diffeomorphic to $\Sigma$. Denote $\partial \tilde{M} = \Sigma_- \cup \Sigma_+$. Let $\tilde{M}_k$, $k \in \Z$ be $\Z$ copies of $\tilde{M}$, and let $\tilde{M}'_k$ be the corresponding $\tilde{M}' \subset \tilde{M}$. Glue together $\tilde{M}_k$ along the boundary by gluing the $\Sigma_+$ boundary component of $\tilde{M}_k$ with the $\Sigma_-$ boundary component of $\tilde{M}_{k+1}$. Denote the resulting manifolds by $$\hat{M} = \cup_{k \in \Z} \tilde{M}_k/ \sim,$$
$$\hat{M}' = \cup_{k \in \Z} \tilde{M}'_k/ \sim,$$
where the equivalence relation $\sim$ is the gluing we just described. Then $\hat{M}$ is an infinite cyclic covering of $M$. 
Denote the closed hypersurface in $\hat{M}$ coming from the $\Sigma_-$ boundary component of $\tilde{M}_{k}$ (equivalently, the $\Sigma_+$ boundary component of $\tilde{M}_{k-1}$) by $\Sigma_k$. Orient $\Sigma_k$ so that its normal is pointing towards $\tilde{M}_{k}$.

Let $X_k'$, $k \in \Z$ be $\Z$ copies of $X'$. Then we have
$$\hat{Y} = \hat{M} \#_\Z X = \hat{M}' \cup (\cup_{k \in \Z} X_k'),$$
where $\hat{M}'$ and each $X_k'$ are glued on the boundary spheres. The manifold $Y$ is an infinite cyclic cover of $M \# X$.

We endow $\hat{Y}$ with the pullback Riemannian metric such that $\hat{Y} \to Y$ is a Riemannian covering map. Then by our assumption that $\cC_m >1$ on $M'$, we also have $\cC_m >1$ on $\hat{M}'$.

\textbf{Step 2: construct the weight function $h$.}

We now define $\rho_0 : \hat{Y} \to \R$ as the signed distance function to the hypersurface $\Sigma_0$. Then $\rho_0$ is Lipschitz. We then take $\rho_1$ to be a smoothing of $\rho_0$ such that for each $k$, $\rho_1 \equiv A_k$ for some constant $A_k$ in a small neighborhood of $\partial X_k'$ (i.e., where $\hat{M}'$ and $X_k'$ are glued together), and $A_k > 0$ if $k \ge 0$ and $A_k < 0$ if $k < 0$. We can further assume that $\rho_1 \ge A_k$ on $X_k'$ if $k \ge 0$ and $\rho_1 \le A_k$ on $X_k'$ if $k < 0$. 

Then there is $L>0$ so that 
\[|\Lip(\rho_1)|_g < \frac{\sqrt{a}}{2}L.\]

We now define a function $h\in C(\hat{Y},[-\infty,\infty])$ as follows. On $\hat{M}' \cap \{ -\tfrac{\pi L}{2} \leq \rho_1 \leq \tfrac{\pi L}{2} \}$, we define 
\[
h(p) = - \frac{1}{\sqrt{a}}\tan(\tfrac{1}{L} \rho_1(p)). 
\]
On the rest of $\hat{M}'$ we set $h = \pm \infty$ such that it is continuous to $[-\infty,\infty]$. We then define $h$ on $X_k'$. When $A_k \ge \tfrac{\pi L}{2}$, set $h=-\infty$ on $X_k'$. When $A_k \le -\tfrac{\pi L}{2}$, set $h=\infty$ on $X_k'$. Now assume $|A_k| \leq \tfrac{\pi L}{2}$.

For $k \ge 0$ and
\[
x \in X_k' \cap \left\{\rho_1 < A_k + \frac{2L}{\tan(L^{-1}A_k)}\right\},
\]

or for $k < 0$ and
\[
x \in X_k' \cap \left\{\rho_1 > A_k + \frac{2L}{\tan(L^{-1}A_k)}\right\},
\]
 we set
\[
h(x) =  \frac{2L}{\sqrt{a}(\rho_1(x) - A_k -  \frac{2L}{\tan ( L^{-1}A_k)})} .
\]
Otherwise we set $h(p)=\pm\infty$ such that $h$ is continuous. Observe that by definition, $h$ is finite on only finitely many $X_k'$.
 
Notice that for $x \in \partial X_k'$, we have that
\[
h(x) = - \frac{1}{\sqrt{a}}\tan (L^{-1}A_k) = - \frac{1}{\sqrt{a}}\tan(L^{-1}\rho_1(x)),
\]
and thus $h$ is Lipschitz across $\partial X_k'$. If $0 \le A_k < \tfrac{\pi L}{2}$, $x\in X_k'$ and 
\[
\rho_1(x) \nearrow A_k + \frac{2L}{\tan(L^{-1}A_k)},
\]
we have that $h(x) \to -\infty$. 
Similarly, if $-\tfrac{\pi L}{2} < A_k < 0$, $x\in X_k'$ and 
\[
\rho_1(x) \searrow A_k + \frac{2L}{\tan(L^{-1}A_k)},
\]
we have that $h(x) \to \infty$. 
Thus $h$ is continuous on $X_k'$.

Note that the set $\{|h| < \infty\}$ is bounded. This is because this region is bounded in $\hat{M}'$, only finitely many ends $X_k'$ are included in this set, and in each $X_k'$, the region where $\{|h|<\infty\}$ is bounded. 

Similar to \cite{soapbubble}, we have
\begin{lem} We can smooth $h$ slightly to find a function $h\in C^\infty(\hat{Y})$ satisfying
\begin{equation}\label{eq:mu-bubble-ineq-h-psc}
(\cC_m)_{\hat{Y}} + ah^2 - 2|D_{\hat{Y}} h| > 0
\end{equation}
on $\{|h|<\infty\}$. 
\end{lem}
\begin{proof}

The function $h$ constructed above is smooth away from $\partial X_k'$ (and Lipschitz there). Since each $\partial X_k'$ is compact and only a finite number of them are contained in $\{|h|<\infty\}$, if we prove \eqref{eq:mu-bubble-ineq-h-psc} for function $h$ considered above, then we can easily find a smooth function satisfying \eqref{eq:mu-bubble-ineq-h-psc}.

Recall $|\nabla(\rho_1)| < \frac{\sqrt{a}}{2}L$. We first check \eqref{eq:mu-bubble-ineq-h-psc} on $\hat{M}'$. There, $\cC_m > 1$. As such, we have that
\[
\cC_m + ah^2 - 2|D_{\hat{Y}} h| > 1 + \tan^2(L^{-1}\rho_1(p)) - \cos^{-2}(L^{-1}\rho_1(p)) = 0. 
\]
On the other hand, on $X_k'$ (we assume that $k\geq 0$ as the $k<0$ case is similar), we only know that $\cC_m > 0$. Nevertheless, we compute
\begin{align*}
&  \cC_m + ah^2 - 2|D_{\hat{Y}} h| \\
 > & \, 0 + \frac{4L^2}{\left( \rho_1(p) - A_k -  \frac{2L}{\tan ( L^{-1}A_k)}\right)^2} - \frac{L^2}{\left( \rho_1(p) - A_k -  \frac{2L}{\tan ( L^{-1}A_k)}\right)^2} \\
> & \, 0. 
\end{align*}

This completes the proof. 
\end{proof}

\textbf{Step 3: apply the $\mu$-bubble technique.}

We consider $\mu$-bubbles with respect to the smooth function $h$ we have just defined. We fix 
\[
\Omega_0 : =(\cup_{k<0} \tilde{M}'_k) \cup (\cup_{k<0}X_k').
\]
We can minimize $$\cA(\Omega)=\cH^{n-1}(\partial^* \Omega) - \int_M (\chi_\Omega-\chi_{\Omega_0})h \, d\cH^n$$ among all Cacioppoli sets $\Omega$ such that $\Omega\Delta \Omega_0$ is compactly contained in $\{|h|<\infty\}$ by Proposition~\ref{prop: existence.regularity}. Denote by $\Omega$ the connected component of the minimizer containing $\{\rho_1=-\frac{\pi L}{2}\}$. Since $n\le 7$, each component of $\partial\Omega$ is compact and regular. By the first variation formula from Lemma~\ref{lem: 1st-var} and the stability inequality for $\cA$ from Lemma~\ref{lem: 2nd-var}, we see that $\Lambda = \partial \Omega$ satisfies 
$H=h$ and 
\begin{equation}\label{eq:def-to-psc-12}
\cQ(\psi) 
= \int_\Lambda \left(|D_\Lambda \psi|^2  - \big(|\sff_{\Lambda}|^2 + \Ric_{\hat{Y}}(\nu,\nu) + \langle D_{\hat{Y}} h, \nu \rangle\big)\psi^2  
\right)d\cH^{n-1} \ge 0
\end{equation}
for all $\psi \in C^\infty(\Lambda)$. 

We can find a compact region $Y' \subset Y$ with smooth boundary so that $\partial\Omega \subset Y'$. Furthermore, we can arrange that $\partial Y' \cap \hat M = \Sigma_{I} \cup \Sigma_{-I}$, for some large $I \in \N$. Note that the other boundary components of $Y'$ thus lie completely in some $X_k'$. 

In particular, 
$\partial Y' \setminus \hat M$ bounds some compact manifold with boundary. Cap these components off and then glue the hypersurfaces $\Sigma_I$ and $\Sigma_{-I}$ to each other. We thus obtain a manifold $\tilde Y$ diffeomorphic to $M' \#_i \tilde X_i$, where $M'$ is a $2I$-cyclic covering of $M$ obtained by cutting and pasting along $\Sigma$, and each $\tilde X_i$ is closed and we have finitely many of them. We also have a hypersurface $\Lambda^{n-1}\subset \tilde Y$ homologous to $[\Sigma^{n-1}\times \{*\}]\in H_{n-1}(\tilde Y)$ that satisfies $H=h$ and \eqref{eq:def-to-psc-12}. We can make $h$ to be a smooth function on $\tilde{Y}$ that agrees with our old $h$ in a neighborhood of $\Lambda$, so that (\ref{eq:mu-bubble-ineq-h-psc}) is satisfied. We can also construct a metric on $\tilde Y$ such that it is isometric to the original metric on $Y$ in a neighborhood of $\Lambda$. Since $Y$ has positive $m$-intermediate curvature, this means 
$\tilde Y$ has positive $m$-intermediate curvature in a neighborhood of $\Lambda$. This also means that on $\Lambda$, we have 
$$H_\Lambda = h,$$
$$(\cC_m)_{\tilde{Y}} + ah^2 - 2|D_{\tilde{Y}} h| > 0,$$
and
$$\cQ(\psi) 
= \int_\Lambda \left(|D_\Lambda \psi|^2  - \big(|\sff_{\Lambda}|^2 + \Ric_{\tilde{Y}}(\nu,\nu) + \langle D_{\tilde{Y}} h, \nu \rangle\big)\psi^2 
\right)d\cH^{n-1} \ge 0
$$
for all $\psi \in C^\infty(\Lambda)$. 
\end{proof}

\section{Proof of Theorem~\ref{thm: SYS}}

We begin this section by proving some simple facts about SYS manifolds. We first give an equivalent definition of an SYS manifold. This is the definition given in e.g. \cite{gromov2019fourlectures} and \cite{lesourd2020positive}.

\begin{lem}\label{lem: SYS, equiv defn}
Let $M^n$ be an orientable closed manifold. Then $M$ being an SYS manifold is equivalent to the following condition: There exists a smooth map $F : M \to \T^{n-2}$, such that the homology class of the pullback of a regular value, $[F^{-1}(t)] \in H_2(M)$,
is non-spherical.
\end{lem}
\begin{proof}
Since the space $S^1$ is a $K(\Z,1)$, we have $[M, S^1] = H^1(M; \Z)$, and the bijection is given by $f \mapsto (f_* : H_1(M) \to H_1(S^1) \cong \Z)$. Thus for any $\beta \in H^1(M;\Z)$ we can get a smooth map $f: M \to S^1$ and vice versa. Further, the preimage of any regular value of $f$ represents the Poincar\'e dual of $\beta$. Thus given $\beta_1,\dots \beta_{n-2} \in H^1(M;\Z)$ we can get a smooth map $F=(F_1,\dots,F_{n-2}): M \to \T^{n-2}$ and vice versa. Since the cup product is the Poincar\'e dual to intersection,  we have 
$$[M] \frown (\beta_1 \smile \beta_2 \smile \dots \smile \beta_{n-2}) = [F_1^{-1}(t_1) \cap \dots \cap F_{n-2}^{-1}(t_{n-2})] = [F^{-1}(t)],$$
where $t=(t_1,\dots, t_{n-2})$ is any regular value of $F$. Then the assertion follows.

\end{proof}
In \cite[Section 5]{gromov2018metric}, Gromov gave some examples of SYS manifolds. For example, we can directly verify that if a closed orientable $n$-manifold admits a map to $\T^n$ of non-zero degree, then it is SYS. Here we establish some simple ways to obtain new SYS manifolds from an old one.

\begin{lem}[{\cite[Section 5, Example 3]{gromov2018metric}}]\label{lem: SYS, degree 1}
Let $M^n$ be an SYS manifold and let $\hat{M}$ be a closed orientable $n$-manifold such that there exists a map 
$f:\hat{M} \to M$ of degree 1. Then $\hat{M}$ is also an SYS manifold.
\end{lem}
\begin{proof}
Since $M$ is an SYS manifold, 
there are nonzero cohomology
classes $\beta_1, \dots, \beta_{n-2}$ in $H^1(M;\Z)$ such that the homology class
$[M] \frown (\beta_1  \smile \dots \smile \beta_{n-2}) \in H_2(M;\Z)$ is non-spherical. 

Then we get pullbacks $f^*\beta_1, \dots, f^*\beta_{n-2}$ in $H^1(\hat M;\Z)$. Claim the class $f^*\beta_1 \smile \dots \smile f^*\beta_{n-2}  \in H_2(\hat{M};\Z)$ is non-spherical.
 Suppose not. Then there exists a map $\phi:S^2 \to \hat M$ such that 
$$[\hat{M}] \frown (f^*\beta_1 \smile \dots \smile f^*\beta_{n-2}) = [\phi(S^2)].$$

By naturality of the cup product and the cap product, and using the fact that $f$ is of degree 1, we have
\begin{align*}
[(f\phi)_*(S^2)]  & = f_*[\phi(S^2)] \\
& = f_*\Big([\hat M] \frown (f^*\beta_1 \smile \dots \smile f^*\beta_{n-2})\Big) \\
& = f_*\Big([\hat M] \frown f^*(\beta_1 \smile \dots \smile \beta_{n-2})\Big) \\
& = f_*[\hat M] \frown (\beta_1 \smile \dots \smile \beta_{n-2}) \\
& = [M] \frown (\beta_1 \smile \dots \smile \beta_{n-2}),
\end{align*}
which means the class $[M] \frown (\beta_1 \smile \dots \smile \beta_{n-2})$ is spherical, contradicting our assumptions. This contradiction shows that $[\hat M] \frown (f^*\beta_1 \smile \dots \smile f^*\beta_{n-2}) \in H_2(\hat M;\Z)$ is also non-spherical. Thus $\hat M$ is an SYS manifold as desired.
\end{proof}

Unlike the case in the previous lemma, if a closed orientable manifold $X$ admits a map $f$ of degree $d>1$ to an SYS manifold, then $X$ is not necessarily SYS \cite[Section 5, Example 3]{gromov2018metric}. What we have instead is the following.

\begin{lem} \label{lem: SYS, cyclic cover}
Suppose $M^n$ is a connected SYS manifold with $\beta_1, \beta_2, \dots, \beta_{n-2}$ in $H^1(M;\Z)$ such that $[M] \frown (\beta_1 \smile \beta_2 \smile \dots \smile \beta_{n-2}) \in H_2(M;\Z)$ is non-spherical and such that the Poincar\'e dual of $\beta_1$ is represented by a closed connected embedded orientable hypersurface $\Sigma$. Let $\hat{M}$ be the $d$-cyclic cover of $M$ obtained by cutting and pasting along $\Sigma$. Then $\hat{M}$ is also an SYS manifold.
\end{lem}
\begin{proof}
Notice that $\Sigma$ is homological nontrivial, hence non-separating. Let $p: \hat{M} \to M$ be the covering map. Let $\Sigma_0$ be one copy of $\Sigma$ in $\hat{M}$, which is also non-separating.  
Let $\hat{\beta}_1 \in H^1(\hat{M};\Z)$ be the Poincar\'e dual of $\Sigma$.
Using naturality of the cup product and the cap product and the assumptions $[\Sigma_0] = [\hat M] \frown \hat{\beta}_1$, $[\Sigma] = [M] \frown \beta_1$, we have
\begin{align*}
& p_*\Big([\hat{M}] \frown (\hat{\beta}_1 \smile p^*\beta_2 \smile \dots \smile p^*\beta_{n-2})\Big) \\
&=p_*\Big(([\hat{M}] \frown \hat{\beta}_1) \frown p^*(\beta_2 \smile \dots \smile \beta_{n-2})\Big) \\
&=p_*\Big([\Sigma_0]\frown p^*(\beta_2 \smile \dots \smile \beta_{n-2})\Big) \\
&=p_*[\Sigma_0]\smile (\beta_2 \smile \dots \smile \beta_{n-2})\Big) \\
&=[\Sigma]\frown (\beta_2 \smile \dots \smile \beta_{n-2}) \\
&=([M] \frown \beta_1)\frown (\beta_2 \smile \dots \smile \beta_{n-2}) \\
&=[M] \frown (\beta_1 \smile \beta_2 \smile \dots \smile \beta_{n-2}).
\end{align*}
Since $[M] \frown (\beta_1 \smile \beta_2 \smile \dots \smile \beta_{n-2}) \in H_2(M;\Z)$ is non-spherical, this means the class $[\hat{M}] \frown (\hat{\beta}_1 \smile p^*\beta_2 \smile \dots \smile p^*\beta_{n-2}) \in H_2(\hat M; \Z)$ is non-spherical as well. Thus $\hat M$ is an SYS manifold as desired. 
\end{proof}

\begin{lem}\label{lem: SYS, hypersurface}
Let $M^n$ be an SYS manifold with nonzero cohomology
classes $\beta_1, \dots, \beta_{n-2}$ in $H^1(M;\Z)$ such that the homology class
$[M] \frown (\beta_1  \smile \dots \smile \beta_{n-2}) \in H_2(M;\Z)$ is non-spherical. Let $\Sigma^{n-1}$ be a closed embedded orientable hypersurface representing the Poincar\'e dual of $\beta_1$. Then $\Sigma$ is an SYS manifold. 
\end{lem}
\begin{proof}
Consider the embedding $f:\Sigma \to M$.
Using naturality of cup product and cap product and the assumption $f_*[\Sigma] = [M] \frown \beta_1$, we have
\begin{align*}
& f_*\Big([\Sigma] \frown (f^*\beta_2 \smile \dots \smile f^*\beta_{n-2})\Big) \\
&= f_*([\Sigma]) \frown (\beta_2 \smile \dots \smile \beta_{n-2}) \\
&= ([M] \frown \beta_1) \frown (\beta_2 \smile \dots \smile \beta_{n-2}) \\
&= [M] \frown (\beta_1 \smile \beta_2 \smile \dots \smile \beta_{n-2}),
\end{align*}
so the class $[\Sigma] \frown (f^*\beta_2 \smile \dots \smile f^*\beta_{n-2}) \in H_2(\Sigma)$ is non-spherical as well. Thus $\Sigma$ is an SYS manifold as desired. 
\end{proof}

We are now ready to give a proof of Theorem~\ref{thm: SYS}.

\begin{proof}[Proof of Theorem~\ref{thm: SYS}]
Assume $M$ is an SYS manifold, $X$ is any closed $n$-manifold, and $M \# X$ admits a complete metric of positive scalar curvature. By taking a connected component we can assume $M$ is connected.

Let $\beta_1, \beta_2, \dots, \beta_{n-2}$ in $H^1(M;\Z)$ be the cohomology classes as in the definition of an SYS manifold. By Lemma~\ref{lem: representable}, we can take $\Sigma \subset M$ to be a closed embedded orientable hypersurface such that $[\Sigma] \in H_{n-1}(M;\Z)$ is dual to $\beta_1$. Then there exists a connected component $\Sigma'$ of $\Sigma$ such that if we denote the Poincar\'e dual of $[\Sigma']$ by $\beta'_1 \in H^1(M;\Z)$, then the homology class $[M] \frown (\beta'_1 \smile \beta_2 \smile \dots \smile \beta_{n-2}) \in H_2(M;\Z)$ is also non-spherical. Then by replacing $\beta_1$ by $\beta'_1$ and $\Sigma$ by $\Sigma'$, we can take $\Sigma$ to be a connected hypersurface dual to $\beta_1$.

Apply Theorem~\ref{thm: main construction} with $m=n-1$ and $a = 1$. Then $m$-intermediate curvature reduces to scalar curvature and we have $\cC_{n-1} =\frac{1}{2} R$. We obtain a closed connected orientable Riemannian manifold $(\tilde Y, \tilde g)$, a smooth function $h \in C^\infty(Y)$, and a closed embedded orientable hypersurface $\Lambda^{n-1} \subset \tilde{Y}$ such that
\begin{enumerate}[(i)]
\item $\tilde Y = M' \#_i \tilde{X}_i$, where $M'$ is a finite cyclic covering of $M$ obtained by cutting and pasting along $\Sigma$ and the $\tilde{X}_i$'s are a finite number of closed manifolds.
\item In a neighborhood of $\Lambda$, $\tilde Y$ has positive scalar curvature.
\item $p_*[\Lambda] = [\Sigma] \in H_{n-1}(M')$, where $p: \tilde Y \to M'$ is the projection map and $[\Sigma]$ is the homology class represented by any copy of $\Sigma$ in $M'$.
\item On $\Lambda$, we have 
$$H_\Lambda = h,$$
$$\frac{1}{2}R_{\tilde{Y}} + h^2 - 2|D_{\tilde{Y}} h| > 0,$$
and 
$$\cQ(\psi) 
= \int_\Lambda \left(|D_\Lambda \psi|^2  - \big(|\sff_{\Lambda}|^2 + \Ric_{\tilde{Y}}(\nu,\nu) + \langle D_{\tilde{Y}} h, \nu \rangle\big)\psi^2 
\right)d\cH^{n-1} \ge 0
$$
for all $\psi \in C^\infty(\Lambda)$. 
\end{enumerate}
Using conditions (i) and (iii) and Lemmas \ref{lem: SYS, cyclic cover}, \ref{lem: SYS, degree 1}, \ref{lem: SYS, hypersurface}, we have that $\Lambda$ is an SYS manifold. 

On the other hand, the traced Gaussian equation gives 
$$R_{\tilde{Y}} = R_\Lambda + 2  \operatorname{Ric}_{\tilde{Y}}(\nu,\nu)+ |\sff_\Lambda|^2 - H_\Lambda^2.$$
Applying this in (iv), we obtain that 
\begin{align*}
\int_\Lambda|D_\Lambda \psi|^2 d\cH^{n-1} &\ge \int_\Lambda  \big(|\sff_{\Lambda}|^2 + \Ric_{\tilde{Y}}(\nu,\nu) + \langle D_{\tilde{Y}} h, \nu \rangle\big)\psi^2 
d\cH^{n-1} \\    
& = \int_\Lambda  \big(|\sff_{\Lambda}|^2 + \frac{1}{2}(R_{\tilde{Y}} - R_\Lambda -|\sff_\Lambda|^2 + h^2) + \langle D_{\tilde{Y}} h, \nu \rangle\big) \psi^2
d\cH^{n-1} \\
& \ge \frac{1}{2} \int_\Lambda (R_{\tilde{Y}}  + h^2 - 2|D_{\tilde{Y}} h|\big)\psi^2 d\cH^{n-1} 
- \frac{1}{2}\int_\Lambda R_\Lambda \psi^2 d\cH^{n-1},
\end{align*}
so by (ii) and (iv), 
$$\int_\Lambda(2|D_\Lambda \psi|^2 + R_\Lambda \psi^2 ) d\cH^{n-1} \ge \int_\Lambda (R_{\tilde{Y}}  + h^2 - 2|D_{\tilde{Y}} h|\big)\psi^2 d\cH^{n-1}  > 0$$
for all $0 \neq \psi \in C^\infty(\Lambda)$.

Since $4\frac{n-1}{n-2} \ge 2$ for $n \ge 3$, this shows the conformal Laplacian $L = - 4\frac{n-1}{n-2}\Delta_\Lambda + R_\Lambda$ has positive first eigenvalue. If we let $\phi>0$ denote the first eigenfunction and $g_\Lambda$ denote the induced metric of $\Lambda$, then $(\Lambda, \phi^{\frac{4}{n-2}}g_\Lambda)$ has scalar curvature ${\tilde {R}}=\phi^{-(n+2)/(n-2)}L\phi > 0$.

This is a contradiction because by \cite{SY:descent}, an SYS manifold of dimension $3 \le n \le 7$ cannot admit a metric of positive scalar curvature.
\end{proof}

\section{Proof of Theorem~\ref{thm: intermediate curvature}}
\subsection{Modified stable weighted slicings}
In this subsection, we closely follow Section 3 of \cite{brendle2022generalization}. We modify the construction of stable weighted slicing given there and define the modified stable weighted slicing as follows. The only difference is how we define the top slice $\Sigma_1$.
For a stable weighted slicing, $\Sigma_1$ is a stable minimal hypersurface of $\Sigma_0$; in comparison, we require $\Sigma_1$ to come from the boundary component of some $\mu$-bubble. In particular, $\Sigma_1$ is the same type of hypersurface that we obtain from Theorem~\ref{thm: main construction}.
Our goal is to show that positive $m$-intermediate curvature obstructs the existence of modified stable weighted slicings.

\begin{defn}[Modified stable weighted slicing of order $m$ with constant $a$] \ \\
Suppose $2 \leq m \leq n-1$ and let
$(N^n,g)$ be a Riemannian manifold of dimension $\dim N = n$. A modified stable weighted slicing of order $m$ with constant $a>0$ consists of a collection of submanifolds $\Sigma_k$, $0 \leq k \leq m$, a smooth function $h \in C^\infty(N)$, and a collection of positive functions $\rho_k \in C^\infty(\Sigma_k)$ satisfying the following conditions: 
\begin{itemize}
\item $\Sigma_0 = N$ and $\rho_0 = 1$.
\item For $k=1$, $\Sigma_1$ is an 
embedded two-sided hypersurface in $\Sigma_0$ such that  
\begin{itemize}
\item the mean curvature satisfies $H_{\Sigma_1} = h,$
\item the operator $\cL_1 = -\Delta_{\Sigma_1} -|\sff_{\Sigma_1}|^2 - \Ric_{\Sigma_0}(\nu_1,\nu_1) - \langle D_{\Sigma_0} h, \nu_1 \rangle$ is a non-negative operator, where $\nu_1$ is a unit normal vector field along $\Sigma_1$,
\item we have $(\cC_m)_{\Sigma_0} + aH_{\Sigma_1}^2 - |D_{\Sigma_0} h| > 0,$
on $\Sigma_1$.
\end{itemize}

\item For each $2 \leq k \leq m$, $\Sigma_k$ is an embedded two-sided hypersurface in $\Sigma_{k-1}$. Moreover, $\Sigma_k$ is a stable critical point of the $\rho_{k-1}$-weighted area
 \[
  \mathcal{H}^{n-k}_{\rho_{k-1}}(\Sigma) = \int_\Sigma \rho_{k-1} \, d\mu
 \]
in the class of hypersurfaces $\Sigma \subset \Sigma_{k-1}$.
\item For $k=1$, $v_1 = \rho_1 \in C^\infty(\Sigma_1)$ is a first eigenfunction of $\cL_1$. For each $2 \leq k \leq m$, the function $v_k = \frac{\rho_k}{\rho_{k-1}|_{\Sigma_k}} \in C^\infty(\Sigma_k)$ is a first eigenfunction of the stability operator associated with the $\rho_{k-1}$-weighted area. 
\end{itemize}
 \label{defn: modified stable weighted slicing}
\end{defn}
Let $(N^n, g)$ be a closed Riemannian manifold of dimension $n$. Throughout this subsection, we assume that we are given a modified stable weighted slicing of order $m$. Then all the calculations in \cite[Section 3]{brendle2022generalization} for $\Sigma_k$, $2 \le k \le m$ carry over, and we record them here.

By the first variation formula for weighted area, Corollary
 2.2 in \cite{brendle2022generalization},
 the mean curvature $H_{\Sigma_k}$ of the slice $\Sigma_{k}$ in 
 the manifold $\Sigma_{k+1}$
 satisfies for $2 \leq k \leq m$ the relation 
 \[
  H_{\Sigma_k} = - \langle D_{\Sigma_{k-1}} \log \rho_{k-1}, \nu_k \rangle.
 \]

 By the second variation formula for weighted area, Proposition 2.3 in \cite{brendle2022generalization},
 we obtain for $2 \leq k \leq m$
 the inequality
 \begin{align*}
  0 \leq& \int_{\Sigma_{k}}
  \rho_{k-1}
  \left(
  - \psi \Delta_{\Sigma_k} \psi
  - \psi \langle D_{\Sigma_k} \log \rho_{k-1}, D_{\Sigma_k}
  \psi \rangle
  \right) \, d\mu \\
  &- 
  \int_{\Sigma_k}
  \rho_{k-1}
  \left( 
  |\sff_{\Sigma_k}|^2 + \Ric_{\Sigma_{k-1}}(\nu_k, \nu_k)
  - (D_{\Sigma_{k-1}}^2 \log \rho_{k-1})(\nu_k, \nu_k)
  \right) \psi^2
  \, d\mu
 \end{align*}
 for all $\psi \in C^{\infty}(\Sigma_k)$. By Definition \ref{defn: modified stable weighted slicing} we may write $\rho_k = \rho_{k-1} \, v_k$, where $v_k > 0$ is the first eigenfunction of the stability operator for the weighted area functional on $\Sigma_k$. The function $v_k$ satisfies
   \begin{align*}
  \lambda_k v_k 
  =&
   - \Delta_{\Sigma_k} v_k 
   - \langle D_{\Sigma_k} \log \rho_{k-1}, D_{\Sigma_k} v_k \rangle
  - \left( 
  |\sff_{\Sigma_k}|^2 + \Ric_{\Sigma_{k-1}}(\nu_k, \nu_k) \right) v_k \\
  &+ (D_{\Sigma_{k-1}}^2 \log \rho_{k-1})(\nu_k, \nu_k)
   v_k,
 \end{align*}
where $\lambda_k \geq 0$ denotes the first eigenvalue of the stability operator.  

 By setting $w_k = \log v_k$ we record the following equation:
 \begin{equation}
 \begin{aligned}
\lambda_k =& - \Delta_{\Sigma_k} w_k - \langle D_{\Sigma_k} \log \rho_{k-1}, D_{\Sigma_k} w_k \rangle - \left ( 
  |\sff_{\Sigma_k}|^2 + \Ric_{\Sigma_{k-1}}(\nu_k, \nu_k) \right ) \\
  &+ (D_{\Sigma_{k-1}}^2 \log \rho_{k-1})(\nu_k, \nu_k)
  - |D_{\Sigma_k}  w_k|^2.
  \end{aligned}
  \label{equation:LograrithmicWeights}
\end{equation}

\begin{lem}[First slicing identity, {\cite[Lemma 3.1]{brendle2022generalization}}] \ \\
We have for $2 \leq k \leq m$ the identity
\begin{align*}
 \Delta_{\Sigma_k} \log \rho_{k-1} + (D_{\Sigma_{k-1}}^2 \log \rho_{k-1}) (\nu_k, \nu_k)
 =
 \Delta_{\Sigma_{k-1}} \log \rho_{k-1} + H_{\Sigma_k}^2.
\end{align*}
\label{lem: FirstSlicingEquality}
\end{lem}

\begin{lem}[Second slicing identity, {\cite[Lemma 3.2]{brendle2022generalization}}] \ \\
We have for $2 \leq k \leq m-1$
the identity
\begin{align*}
 \Delta_{\Sigma_{k}} \log \rho_{k}
  =& \Delta_{\Sigma_{k}} \log \rho_{k-1}
 +
   (D_{\Sigma_{k-1}}^2 \log \rho_{k-1})(\nu_k, \nu_k) \\
 &-
 \left(
 \lambda_{k}
 + |\sff_{\Sigma_{k}}|^2
 + \Ric_{\Sigma_{k-1}}(\nu_k, \nu_k)
 +
\langle D_{\Sigma_{k}} \log \rho_{k}, D_{\Sigma_{k}} w_{k} \rangle
 \right).
\end{align*}
\label{lem: SecondSlicingEquality}
\end{lem}

\begin{lem}[Second slicing identity for $k=1$] \ \\
We have the identity
\begin{align*}
 \Delta_{\Sigma_{1}} \log \rho_{1}
  =& - \left(
 \lambda_{1}
 + |\sff_{\Sigma_{1}}|^2
 + \Ric_{\Sigma_{0}}(\nu_1, \nu_1)
 +
\langle D_{\Sigma_{1}} \log \rho_{1}, D_{\Sigma_{1}} w_{1} \rangle - \langle D_{\Sigma_0} h, \nu_1 \rangle
 \right).
\end{align*}
\label{lem: SecondSlicingEquality, k =1}
\end{lem}
\begin{proof}
This is a direct computation using that $\rho_1$ is a first eigenfunction of $\cL_1$ with eigenvalue $\lambda_1$ and $w_1 = \log \rho_1$. 
\end{proof}

\begin{lem}[Stability inequality on the bottom slice, {\cite[Lemma 3.3]{brendle2022generalization}}] \ \\
On the bottom slice $\Sigma_m$
 we have the inequality
 \begin{align*}
  \int_{\Sigma_m}
  \rho_{m-1}^{-1}
  \left(
  \Delta_{\Sigma_{m-1}} \log \rho_{m-1}
  + H_{\Sigma_m}^2 \right) d\mu \geq
\int_{\Sigma_m} \rho_{m-1}^{-1}  \left(
|\sff_{\Sigma_{m}}|^2 + \Ric_{\Sigma_{m-1}}(\nu_{m}, \nu_{m})
  \right) 
  \, d\mu.
\end{align*}
\label{lem: StabilityInequality_BottomSlice}
\end{lem}

Similar to {\cite[Lemma 3.4]{brendle2022generalization}}, we have the following:
\begin{lem}[Main inequality]  \ \\
We have the inequality
\begin{align*}
    \int_{\Sigma_m} \rho_{m-1}^{-1}
    \left( 
    \Lambda + \mathcal{R} + \mathcal{G} + \mathcal{E} + \langle D_{\Sigma_0} h, \nu_1 \rangle
    \right) \, d\mu \leq 0,
\end{align*}
where the eigenvalue term $\Lambda$,
the intrinsic curvature term $\mathcal{R}$,
the extrinsic curvature term $\mathcal{E}$,
and the gradient term $\mathcal{G}$ are given by
\begin{align*}
\Lambda
&=\sum_{k=1}^{m-1} \lambda_k, \;
\mathcal{R}
=\sum_{k=1}^m\Ric_{\Sigma_{k-1}}(\nu_k,\nu_k), \;
\mathcal{G}
=\sum_{k=1}^{m-1}
\langle
 D_{\Sigma_k} \log \rho_k, D_{\Sigma_k} w_k
\rangle, \\
\; \text{and} \; \;
\mathcal{E}
&=\sum_{k=1}^m |\sff_{\Sigma_k}|^2 - \sum_{k=2}^m H_{\Sigma_k}^2.
\end{align*}
\label{Lemma: MainInequality} 
\end{lem}
\begin{proof}
If we combine
Lemma \ref{lem: FirstSlicingEquality},
and 
Lemma \ref{lem: SecondSlicingEquality},
we obtain for $2 \leq k \leq m-1$
the identity
\begin{align*}
    \Delta_{\Sigma_k} \log \rho_k
 = \Delta_{\Sigma_{k-1}} \log \rho_{k-1}
 +
  H_{\Sigma_{k}}^2
  &-
 \left(
 \lambda_{k}
 + |\sff_{\Sigma_{k}}|^2
 + \Ric_{\Sigma_{k-1}}(\nu_k, \nu_k)
 +
\langle D_{\Sigma_{k}} \log \rho_{k}, D_{\Sigma_{k}} w_{k} \rangle
 \right).
\end{align*}

Summation of this formula over $k$ from $2$ to $m-1$ and using Lemma~\ref{lem: SecondSlicingEquality, k =1} yields 
\begin{align*}
    \Delta_{\Sigma_m-1} \log \rho_{m-1}
 =& - \langle D_{\Sigma_0} h, \nu_1 \rangle
 +\sum_{k=2}^{m-1} H_{\Sigma_k}^2  \\
 &-
 \sum_{k=1}^{m-1}
   \left(
  \lambda_{k}
  + |\sff_{\Sigma_{k}}|^2
  + \Ric_{\Sigma_{k-1}}(\nu_{k}, \nu_{k})
  + 
 \langle D_{\Sigma_{k}} \log \rho_{k}, D_{\Sigma_{k}} w_{k} \rangle
  \right).
\end{align*}
Plugging this into Lemma~\ref{lem: StabilityInequality_BottomSlice} yields the result.
\end{proof}

The eigenvalue term $\Lambda$ is non-negative, since it is the sum of the non-negative eigenvalues. To estimate the other terms in the above lemma, fix a point $x \in \Sigma_m$ and consider an orthonormal basis $\{e_1,\dots, e_n\}$ of $T_xN$ with $e_j = \nu_j$ for $1 \le j \le m$ as above. We define for each $1 \leq k \leq m$ the extrinsic curvature terms $\mathcal{V}_k$:
\begin{align*}
 \mathcal{V}_1
 =&
  |\sff_{\Sigma_1}|^2
  + 
   \sum_{p = 2}^{m} \sum_{q = p+1}^n
  \left(
   \sff_{\Sigma_1}(e_p, e_p) \sff_{\Sigma_1}(e_q, e_q)
   -
   \sff_{\Sigma_1}(e_p, e_q)^2 
  \right)
   ,
  \\
  \mathcal{V}_k =&
  |\sff_{\Sigma_k}|^2
  -  \left( \frac{1}{2} - \frac{1}{2(k-1)} \right)
 H_{\Sigma_k}^2 \\
 &+
\sum_{p = k+1}^{m} \sum_{q = p+1}^n
  \left(
   \sff_{\Sigma_k}(e_p, e_p) \sff_{\Sigma_k}(e_q, e_q)
   -
   \sff_{\Sigma_k}(e_p, e_q)^2 
  \right)
   \; \text{for} \; 2 \leq k \leq m - 1, \\
   \mathcal{V}_m
   =&
   |\sff_{\Sigma_m}|^2
    -  \left( \frac{1}{2} - \frac{1}{2(m-1)} \right)
 H_{\Sigma_m}^2.
\end{align*}

Inspecting the estimate for $\cG$ and $\cR$ in \cite[Lemma 3.7 and 3.8]{brendle2022generalization}, we see that the same calculations carry over, so we have the following lemma:
\begin{lem}{\cite[Lemma 3.10]{brendle2022generalization}} \ \\
We have the pointwise estimate on $\Sigma_m$:
\begin{align*}
 \mathcal{R} + \mathcal{E} + \mathcal{G}
 \geq \mathcal{C}_{m}(e_1, \dots, e_m) 
 + \sum_{k=1}^m \mathcal{V}_k.
 \end{align*}
 \label{lem: Reduced-MainInequality}
\end{lem}

Now we need to estimate the extrinsic curvature terms $\cV_k$. The estimate on the top slice is what differs from \cite{brendle2022generalization}.

\begin{lem}[Extrinsic curvature terms on the top slice] \ \label{lem: extrinsic on top slice}\\
Suppose $m^2 - mn + 2n -2 > 0$ and $m^2 - mn + m + n > 0$. Then 
we have the estimate
\begin{align*}
    \mathcal{V}_1
    \geq a H_{\Sigma_1}^2 \ge 0,
\end{align*}
for any $0 \le a \le \min\{\frac{m}{2(m-1)}, \frac{1}{n-m}, \frac{m^2 - mn + m + n}{2(m^2 - mn + 2n -2)}\}$.
\label{lem: ExtrinsicCurvature_TopSlice}
\end{lem}
\begin{proof}
Consider the quantity $\cV_1 - aH_{\Sigma_1}^2$ for some $a$ satisfying $$0 \le a \le \min\{\frac{m}{2(m-1)}, \frac{1}{n-m}, \frac{m^2 - mn + m + n}{2(m^2 - mn + 2n -2)}\}.$$
We begin by discarding the off-diagonal terms of the second fundamental form $\sff_{\Sigma_1}$:
\begin{align*}
& \cV_1 - aH_{\Sigma_1}^2 \\
=& \,   |\sff_{\Sigma_1}|^2
  + 
   \sum_{p = 2}^{m} \sum_{q = p+1}^n
  \left(
   \sff_{\Sigma_1}(e_p, e_p) \sff_{\Sigma_1}(e_q, e_q)
   -
   \sff_{\Sigma_1}(e_p, e_q)^2 
  \right) - aH_{\Sigma_1}^2 \\
\ge & \,\sum_{p=2}^n \sff_{\Sigma_1}(e_p,e_p)^2
  + 
   \sum_{p = 2}^{m} \sum_{q = p+1}^n
   \sff_{\Sigma_1}(e_p, e_p) \sff_{\Sigma_1}(e_q, e_q)
 - aH_{\Sigma_1}^2 \\
=&\, \sum_{p=2}^n \sff_{\Sigma_1}(e_p,e_p)^2
  + 
   \sum_{p = 2}^{m}  \sff_{\Sigma_1}(e_p, e_p)  \sum_{q = p+1}^m
\sff_{\Sigma_1}(e_q, e_q)
  + 
   \sum_{p = 2}^{m}  \sff_{\Sigma_1}(e_p, e_p)  \sum_{q = m+1}^n
\sff_{\Sigma_1}(e_q, e_q)
 - aH_{\Sigma_1}^2 \\
=& \,\frac{1}{2}\sum_{p=2}^m \sff_{\Sigma_1}(e_p,e_p)^2
  + \sum_{q=m+1}^n \sff_{\Sigma_1}(e_q,e_q)^2
 +
   \frac{1}{2}\left(\sum_{p = 2}^{m}  \sff_{\Sigma_1}(e_p, e_p) \right)^2
    \\
   & \,+ 
   \sum_{p = 2}^{m}  \sff_{\Sigma_1}(e_p, e_p)  \sum_{q = m+1}^n
\sff_{\Sigma_1}(e_q, e_q)
 - a\left( \sum_{p = 2}^{m}  \sff_{\Sigma_1}(e_p, e_p) +  \sum_{q=m+1}^n \sff_{\Sigma_1}(e_q,e_q)\right)^2 \\
 = & \,\frac{1}{2}\sum_{p=2}^m \sff_{\Sigma_1}(e_p,e_p)^2
  + \sum_{q=m+1}^n \sff_{\Sigma_1}(e_q,e_q)^2
 +
   \left(\frac{1}{2} - a \right)\left(\sum_{p = 2}^{m}  \sff_{\Sigma_1}(e_p, e_p) \right)^2
     \\
  & \, - a \left(\sum_{q=m+1}^n \sff_{\Sigma_1}(e_q,e_q) \right)^2 + \left( 1 - 2a\right)
   \sum_{p = 2}^{m}  \sff_{\Sigma_1}(e_p, e_p)  \sum_{q = m+1}^n
\sff_{\Sigma_1}(e_q, e_q).
\end{align*}
For simplicity, let $A: = \sum_{p = 2}^{m}  \sff_{\Sigma_1}(e_p, e_p)$ and $B: = \sum_{q = m+1}^{n}  \sff_{\Sigma_1}(e_q, e_q)$. By the Cauchy–Schwarz inequality,
$$\sum_{p=2}^m \sff_{\Sigma_1}(e_p,e_p)^2 \ge \frac{1}{m-1}\left(\sum_{p = 2}^{m}  \sff_{\Sigma_1}(e_p, e_p) \right)^2 =  \frac{1}{m-1} A^2$$
and 
$$\sum_{q=m+1}^n \sff_{\Sigma_1}(e_q,e_q)^2 \ge \frac{1}{n-m}\left(\sum_{q = m+1}^{n}  \sff_{\Sigma_1}(e_q, e_q) \right)^2 = \frac{1}{n-m} B^2.$$
Thus 
\begin{align*}
\cV_1 - aH_{\Sigma_1}^2
\ge & \left(\frac{m}{2(m-1)} - a \right)A^2
+  \left(\frac{1}{n-m} - a \right)B^2 + \left( 1 - 2a\right)
  AB \\
\ge & \left( 2 \sqrt{\left(\frac{m}{2(m-1)} - a \right)\left(\frac{1}{n-m} - a \right)} - (1-2a) \right)|AB|,
\end{align*}
using the assumptions $a \le \frac{m}{2(m-1)}$ and $a \le \frac{1}{n-m}$ and the AM-GM inequality.

Using $a \le \frac{m^2 - mn + m + n}{2(m^2 - mn + 2n -2)}$, we have
\begin{align*}
&  4 \left(\frac{m}{2(m-1)} - a \right)\left(\frac{1}{n-m} - a \right) - (1-2a)^2 \\
= & \,\frac{m^2 - mn + m + n}{(m-1)(n-m)} - \frac{2(m^2 - mn + 2n-2)}{(m-1)(n-m)}a \\
\ge & \, 0,
\end{align*}
so 
$\cV_1 - aH_{\Sigma_1}^2 \ge 0$
as desired.

\end{proof}

Again, the estimate for $2 \le k \le m$ in \cite{brendle2022generalization} carry over, so we have the following two lemmas.
\begin{lem}[Extrinsic curvature terms on intermediate slices, {\cite[Lemma 3.12]{brendle2022generalization}}] \ \\
We have for $2 \leq k \leq m-1$ the estimate
\begin{align*}
    \mathcal{V}_k \geq \frac{m^2-mn+2n-2}{2(m-1)(n-m)}
    \left ( \sum_{q=m+1}^n \sff_{\Sigma_k}(e_q,e_q) \right )^2.
\end{align*}
\label{lem: ExtrinsicCurvature_MidSlices}
\end{lem}

\begin{lem}[Extrinsic curvature terms on the bottom slice, {\cite[Lemma 3.13]{brendle2022generalization}}] \ \\
We have the estimate
\begin{equation*}
    \mathcal{V}_m
    \geq
 \frac{m^2-mn+2n-2}{2(n-m)(m-1)} \, 
 H_{\Sigma_m}^2.
\end{equation*}
\label{lem: ExtrinsicCurvature_BottomSlice}
\end{lem}

We record by direct computation the following lemma:
\begin{lem}[Algebraic lemma] \label{lem: algebraic}\ \\
Suppose $3 \le n \le 7$ and $2 \le m \le n-1$ are integers.
We define the quantity $\xi(n,m) \in \R$ by the formula
$$\xi(n,m) = \min\{m^2-mn+2n-2, m^2-mn+m+n\}.$$
Then for $3 \le n \le 5$, we have
$\xi(n,m) > 0$ for all $2 \le m \le n-1$. For $6 \le n \le 7$, we have
$\xi(n,m) > 0$ precisely when $n-2 \le m \le n-1$.
\end{lem}
\begin{rem} \label{rem: dimension constraint}
The quantity $m^2-mn+2n-2$ is the same as the one in \cite[Lemma 3.14]{brendle2022generalization}. Compared to \cite{brendle2022generalization}, we need the extra constraint $m^2 - mn + m + n > 0$ coming from Lemma~\ref{lem: extrinsic on top slice}. This comes from our requirement for the top slice to come from the boundary component of some $\mu$-bubble. Unlike stable minimal hypersurfaces where $H=0$, in our case we have no \textit{a priori} bound on the mean curvature of the top slice, so we need extra constraint on the dimensions to control it. In \cite{chu2022rigidity}, Chu--Kwong--Lee used $\mu$-bubbles on the bottom slice in their proof of the rigidity result, and therefore needed the same constraint $m^2 - mn + m + n > 0$. This is why their rigidity result is stated for $n \le 5$. Xu \cite{xu2023dimension} proved the same estimate on the bottom slice without using $\mu$-bubbles, and thus extended the rigidity result to $n = 6$. In our case, since we need $\mu$-bubbles to reduce the non-compact setting to a compact setting, it is unclear whether we can get rid of the constraint $m^2 - mn + m + n > 0$.
\end{rem}

Using the above lemmas, we can show that manifolds with positive $m$-intermediate curvature do not allow
stable weighted slicings of order $m$ with constant $a$.

\begin{thm}[$m$-intermediate curvature and modified stable weighted slicings] \ \\
Assume that  $m^2 - mn + 2n -2 > 0$ and $m^2 - mn + m + n > 0$. Assume $0 < a \le \min\{\frac{m}{2(m-1)}, \frac{1}{n-m}, \frac{m^2 - mn + m + n}{2(m^2 - mn + 2n -2)}\}$. Suppose the closed Riemannian manifold $(N^n, g)$ admits a modified stable 
 weighted slicing
 \[
 \Sigma_m \subset \dots \subset \Sigma_1 \subset \Sigma_0 = N^n
 \]
 of order $2 \le m \leq n-1$ with constant $a$. Then we must have $(\cC_m)_N \le 0$ at some point on $\Sigma_m$.
 \label{theorem:SmoothSlicings}
\end{thm}
\begin{proof}
Suppose that the Riemannian manifold $(N^n, g)$ admits a stable weighted slicing
 \[
 \Sigma_m \subset \dots \subset \Sigma_1 \subset \Sigma_0 = N^n
 \]
of order $2 \le m \leq n-1$ with constant $a$, and $(\cC_m)_N > 0$ on $\Sigma_m$.

Combining the estimates for the extrinsic curvature terms, Lemmas \ref{lem: ExtrinsicCurvature_TopSlice}, \ref{lem: ExtrinsicCurvature_MidSlices} and \ref{lem: ExtrinsicCurvature_BottomSlice}, with Lemma \ref{lem: Reduced-MainInequality} implies
\begin{align*}
    \mathcal{R} + \mathcal{E} + \mathcal{G} \geq 
    \mathcal{C}_m(e_1, \dots, e_m) + aH_{\Sigma_1}^2,
\end{align*}
which holds on all points on $\Sigma_m$.
By definition of $\Sigma_1$, the following inequality holds on $\Sigma_1$:
$$(\cC_m)_{\Sigma_0} + aH_{\Sigma_1}^2 - |\langle D_{\Sigma_0} h, \nu_1 \rangle| > 0.$$  
Combining these two inequalities yields that on $\Sigma_m$, we have 
\begin{align*}
    \mathcal{R} + \mathcal{E} + \mathcal{G} > \langle D_{\Sigma_0} h, \nu_1 \rangle.
\end{align*}
This contradicts the main inequality, Lemma \ref{Lemma: MainInequality}. Therefore we must have $(\cC_m)_N \le 0$ at some point on $\Sigma_m$.
\end{proof}
When $m=1$, $m$-intermediate curvature reduces to Ricci curvature, and we also have a non-existence result.

\begin{thm}[Ricci curvature and modified stable weighted slicings] \ \\
Suppose the closed Riemannian manifold $(N^n, g)$ admits a modified stable 
 weighted slicing
 \[
  \Sigma_1 \subset \Sigma_0 = N^n
 \]
 of order $m=1$ with constant $\frac{1}{n-1}$. Then we must have $(\cC_1)_N \le 0$ at some point on $\Sigma_1$.
 \label{theorem:SmoothSlicings, m=1}
\end{thm}
\begin{proof}
Suppose that $(N^n, g)$ admits a stable weighted slicing
 \[
 \Sigma_1 \subset \Sigma_0 = N^n
 \]
of order $m=1$ with constant $a=\frac{1}{n}$, and a metric of positive Ricci curvature. 
By definition, that means we have a smooth function $h \in C^\infty(N)$ such that 
\begin{itemize}
\item the mean curvature of $\Sigma_1$ satisfies $H_{\Sigma_1} = h,$
\item the operator $\cL_1 = -\Delta_{\Sigma_1} -|\sff_{\Sigma_1}|^2 - \Ric_{\Sigma_0}(\nu_1,\nu_1) - \langle D_{\Sigma_0} h, \nu_1 \rangle$ is a non-negative operator, where $\nu_1$ is a unit normal vector field along $\Sigma_1$,
\item we have $(\cC_1)_{\Sigma_0} + \frac{1}{n-1}H_{\Sigma_1}^2 - |\langle D_{\Sigma_0} h, \nu_1 \rangle| > 0,$
on $\Sigma_1$.
\end{itemize}
The second condition means that 
$$\int_{\Sigma_1} |\nabla \phi|^2 \, d\Sigma_1 \ge \int_{\Sigma_1} (|\sff_{\Sigma_1}|^2 + \Ric_{\Sigma_0}(\nu_1,\nu_1) + \langle D_{\Sigma_0} h, \nu_1\rangle)\phi^2 \, d\Sigma_1$$
for any $\phi \in C^\infty(\Sigma_1)$. Here we set $\phi = 1$, and we get 
$$ \int_{\Sigma_1} |\sff_{\Sigma_1}|^2 + \Ric_{\Sigma_0}(\nu_1,\nu_1) + \langle D_{\Sigma_0} h, \nu_1\rangle \, d\Sigma_1 \le 0$$

On the other hand, by discarding the off-diagonal terms and using the Cauchy-Schwarz inequality, we have that on $\Sigma_1$,
\begin{align*}
|\sff_{\Sigma_1}|^2 + \Ric_{\Sigma_0}(\nu_1,\nu_1) + \langle D_{\Sigma_0} h, \nu_1 \rangle & \ge \frac{1}{n-1}H_{\Sigma_1}^2 + \Ric_{\Sigma_0}(\nu_1,\nu_1) + \langle D_{\Sigma_0} h, \nu_1 \rangle \\
& \ge \frac{1}{n-1}H_{\Sigma_1}^2 + \cC_1 - |\langle D_{\Sigma_0} h, \nu_1 \rangle| \\
&> \, 0,
\end{align*}
which contradicts the integral inequality above. 

\end{proof}

\subsection{Existence of modified stable weighted slicings}
In this section we prove the existence of stable weighted slicings of order $m$, thus finishing the proof of Theorem~\ref{thm: intermediate curvature}.

\begin{proof}[Proof of Theorem~\ref{thm: intermediate curvature}]
Assume either $3 \le n \le 5$, $1 \le m \le n-1$ or $6 \le n \le 7$, $m \in \{1, n-2, n-1\}$. Suppose $F: N^n \rightarrow \T^m \times M^{n-m}$ has degree $d \neq 0$. By taking a connected component we can assume $N$ is connected.

 The projection of $F$ onto the factors
 yields maps $f_0: N \rightarrow M$
 and maps $f_1, \dots, f_m: N \rightarrow S^1$.
 Let $\Theta$ be a top-dimensional form of
 the manifold $M$
 normalized such that $\int_M \Theta = 1$,
 and let $\theta$ be a one-form
 on the circle $S^1$ with $\int_{S^1}\theta = 1$.
We define the pull-back forms $\Omega := f_0^* \, \Theta$
and $\omega_j := f_j^* \, \theta$.
By the normalization condition we deduce that
$\int_N \omega_1 \wedge \dots \wedge \omega_m \wedge \Omega = d$. 

 By Lemma~\ref{lem: representable}, we can take $\Sigma \subset M$ to be a closed embedded orientable hypersurface such that $[\Sigma] \in H_{n-1}(M;\Z)$ is dual to $\omega_1$. Then there exists a connected component $\Sigma'$ of $\Sigma$ such that if we denote the Poincar\'e dual of $[\Sigma']$ by $\omega'_1 \in H^1(M;\Z)$, then we have $\int_N \omega'_1 \wedge \dots \wedge \omega_m \wedge \Omega = d'$ for some nonzero $d'$. Then by replacing $\beta_1$ by $\beta'_1$,  $\Sigma$ by $\Sigma'$, $d$ by $d'$, and $f_1$ by a smooth map representing $\beta'_1$, we can take $\Sigma$ to be a connected hypersurface dual to $\beta_1$.

Suppose the manifold $N \# X$ has positive $m$-intermediate curvature. Then we apply Theorem~\ref{thm: main construction} with an arbitrary $a>0$ to be determined later. We obtain a closed orientable Riemannian manifold $(\tilde Y, \tilde g)$, a smooth function $h \in C^\infty(Y)$, and a closed embedded orientable hypersurface $\Lambda_1^{n-1} \subset \tilde{Y}$ such that
\begin{enumerate}[(i)]
\item $\tilde Y = N' \#_i \tilde{X}_i$, where $N'$ is a finite cyclic covering of $N$ obtained by cutting and pasting along $\Sigma$ and the $\tilde{X}_i$'s are a finite number of closed manifolds.
\item In a neighborhood of $\Lambda_1$, $\tilde Y$ has positive $m$-intermediate curvature.
\item $p_*[\Lambda_1] = [\Sigma] \in H_{n-1}(N')$, where $p: \tilde Y \to N'$ is the projection map and $[\Sigma]$ is the homology class represented by any copy of $\Sigma$ in $N'$.
\item On $\Lambda_1$, we have 
$$H_{\Lambda_1} = h,$$
$$R_{\tilde{Y}} + ah^2 - 2|D_{\tilde{Y}} h| > 0,$$
and 
$$\cQ(\psi) 
= \int_\Lambda \left(|D_{\Lambda_1} \psi|^2  - \big(|\sff_{\Lambda_1}|^2 + \Ric_{\tilde{Y}}(\nu,\nu) + \langle D_{\tilde{Y}} h, \nu \rangle\big)\psi^2 
\right)d\cH^{n-1} \ge 0
$$
for all $\psi \in C^\infty(\Lambda_1)$. 
\end{enumerate}

Let $\Lambda_0 = \tilde{Y}$.
Then $\Lambda_1 \subset \Lambda_0 = \tilde{Y}$ gives a modified stable weighted slicing of order $1$ with constant $a$. If $m = 1$, we set $a = \frac{1}{n}$. Then by Theorem~\ref{theorem:SmoothSlicings, m=1}, we must have $(\cC_1)_{\tilde{Y}} \le 0$ on some point of $\Sigma_1$, which contradicts condition (ii). This shows $N \# X$ cannot have positive $1$-intermediate curvature. 

Now assume $m \ge 2$. By condition (i), $\tilde{Y}$ admits a map $G:\tilde{Y} \to N$  with some nonzero degree. By condition (iii) we find that $G_*[\Lambda_1] = [\Sigma] \in H_{n-1}(N)$, so using naturality of the cup and cap products, we obtain
\begin{align*}
& G_*\big([\Lambda_1] \frown ( G^*\omega_2 \smile \dots \smile G^*\omega_m \smile G^*\Omega) \big) \\
& = G_*\big([\Lambda_1] \frown G^*( \omega_2 \smile \dots \smile \omega_m \smile \Omega)\big) \\
& = G_*\big([\Lambda_1] \frown G^*( \omega_2 \smile \dots \smile \omega_m \smile \Omega)\big) \\
& = G_*[\Lambda_1] \frown ( \omega_2 \smile \dots \smile \omega_m \smile \Omega) \\
& = [\Sigma] \frown ( \omega_2 \smile \dots \smile \omega_m \smile \Omega) \\
& = ([N] \frown \omega_1)\frown ( \omega_2 \smile \dots \smile \omega_m \smile \Omega) \\
& = [N]\frown ( \omega_1 \smile \dots \smile \omega_m \smile \Omega) \\
& = d.
\end{align*}
This shows 
\begin{align*}
\int_{\Lambda_1} G^*{\omega}_2 \wedge \dots \wedge G^*{\omega}_m \wedge G^*{\Omega} 
\neq 0.
\end{align*}
Then for $2 \le k \le m$, one can inductively construct the slices $\Lambda_k$ and the weights $\rho_k$, such that $\int_{\Lambda_k} G^*{\omega}_k \wedge \dots \wedge G^*{\omega}_m \wedge G^*{\Omega} 
\neq 0$
holds. For this, we can use the same argument as in \cite[Proof of Theorem 4.5]{SY:sing-PMT} or \cite[Proof of Theorem 1.5]{brendle2022generalization}, where all the details are given. 

We thus obtain a  modified stable weighted slicing $\Lambda_m \subset \dots \subset \Lambda_1 \subset \Lambda_0 = \tilde{Y}$ of order $m$ with constant $a$.
By our assumption on $n$ and $m$, we have $m^2 - mn + 2n -2 > 0$ and $m^2 - mn + m + n > 0$ by Lemma~\ref{lem: algebraic}. Choose $a$ so that $0 < a \le \min\{\frac{m}{2(m-1)}, \frac{1}{n-m}, \frac{m^2 - mn + m + n}{2(m^2 - mn + 2n -2)}\}$. Then by Theorem~\ref{theorem:SmoothSlicings}, we must have $(\cC_m)_{\tilde Y} \le 0$ at some point on $\Lambda_m \subset \Lambda_1$. This contradicts condition (ii), which shows $N \# X$ cannot have positive $m$-intermediate curvature and thereby completes the proof.
\end{proof}

\bibliographystyle{alpha}
\bibliography{main}

\begin{thebibliography}{LUY20}

\bibitem[BHJ22]{brendle2022generalization}
Simon Brendle, Sven Hirsch, and Florian Johne.
\newblock A generalization of {G}eroch's conjecture.
\newblock \url{https://arxiv.org/abs/2207.08617}, 2022.

\bibitem[CKL22]{chu2022rigidity}
Jianchun Chu, Kwok-Kun Kwong, and Man-Chun Lee.
\newblock Rigidity on non-negative intermediate curvature.
\newblock \url{https://arxiv.org/abs/2208.12240}, 2022.

\bibitem[CL20]{soapbubble}
Otis Chodosh and Chao Li.
\newblock Generalized soap bubbles and the topology of manifolds with positive
  scalar curvature.
\newblock \url{https://arxiv.org/abs/2008.11888}, 2020.

\bibitem[GL83]{gromov1983positive}
Mikhael Gromov and H~Blaine Lawson.
\newblock Positive scalar curvature and the dirac operator on complete
  riemannian manifolds.
\newblock {\em Publications Math{\'e}matiques de l'IH{\'E}S}, 58:83--196, 1983.

\bibitem[Gro96]{gromov1996positive}
Mikhael Gromov.
\newblock Positive curvature, macroscopic dimension, spectral gaps and higher
  signatures.
\newblock In {\em Functional Analysis on the Eve of the 21st Century Volume
  II}, pages 1--213. Springer, 1996.

\bibitem[Gro18]{gromov2018metric}
Misha Gromov.
\newblock Metric inequalities with scalar curvature.
\newblock {\em Geometric and Functional Analysis}, 28(3):645--726, 2018.

\bibitem[Gro19]{gromov2019fourlectures}
Misha Gromov.
\newblock Four lectures on scalar curvature.
\newblock \url{https://arxiv.org/abs/1908.10612}, 2019.

\bibitem[HP99]{huisken1999geometric}
Gerhard Huisken and Alexander Polden.
\newblock Geometric evolution equations for hypersurfaces.
\newblock {\em Calculus of variations and geometric evolution problems}, pages
  45--84, 1999.

\bibitem[LUY20]{lesourd2020positive}
Martin Lesourd, Ryan Unger, and Shing-Tung Yau.
\newblock Positive scalar curvature on noncompact manifolds and the {L}iouville
  theorem.
\newblock \url{https://arxiv.org/abs/2009.12618}, 2020.

\bibitem[Sch84]{schoen1984conformal}
Richard Schoen.
\newblock Conformal deformation of a riemannian metric to constant scalar
  curvature.
\newblock {\em Journal of Differential Geometry}, 20(2):479--495, 1984.

\bibitem[Sch89]{schoen1989variational}
Richard Schoen.
\newblock Variational theory for the total scalar curvature functional for
  {R}iemannian metrics and related topics.
\newblock In {\em Topics in calculus of variations}, pages 120--154. Springer,
  1989.

\bibitem[Sch98]{schick1998counterexample}
Thomas Schick.
\newblock A counterexample to the (unstable) {G}romov--{L}awson--{R}osenberg
  conjecture.
\newblock {\em Topology}, 37(6):1165--1168, 1998.

\bibitem[SY79a]{schoen1979proof}
Richard Schoen and Shing-Tung Yau.
\newblock On the proof of the positive mass conjecture in general relativity.
\newblock {\em Communications in Mathematical Physics}, 65(1):45--76, 1979.

\bibitem[SY79b]{SY:descent}
Richard Schoen and Shing-Tung Yau.
\newblock On the structure of manifolds with positive scalar curvature.
\newblock {\em Manuscripta mathematica}, 28(1):159--183, 1979.

\bibitem[SY17]{SY:sing-PMT}
Richard Schoen and Shing-Tung Yau.
\newblock Positive scalar curvature and minimal hypersurface singularities.
\newblock \url{https://arxiv.org/abs/1704.05490}, 2017.

\bibitem[WZ22]{wang2022generalized}
Xiangsheng Wang and Weiping Zhang.
\newblock On the generalized {G}eroch conjecture for complete spin manifolds.
\newblock {\em Chinese Annals of Mathematics, Series B}, 43(6):1143--1146,
  2022.

\bibitem[Xu23]{xu2023dimension}
Kai Xu.
\newblock Dimension constraints in some problems involving intermediate
  curvature.
\newblock \url{https://arxiv.org/abs/2301.02730}, 2023.

\bibitem[Zhu21]{zhu2021width}
Jintian Zhu.
\newblock Width estimate and doubly warped product.
\newblock {\em Transactions of the American Mathematical Society},
  374(2):1497--1511, 2021.

\end{thebibliography}

\end{document}